\documentclass[11pt,oneside]{amsart}
\usepackage{amsmath}
\usepackage{amssymb}
\usepackage{eucal}
\usepackage{amsthm}
\usepackage{amscd}
\usepackage{hyperref}
\usepackage{soul}
\usepackage{url}
\usepackage{skull}
\usepackage{adjustbox}
\usepackage{mathrsfs}
\usepackage{comment}
\usepackage{tikz}
\usetikzlibrary{shapes.geometric, arrows}
\usetikzlibrary{shapes,arrows}
\usepackage[latin1]{inputenc}
\usepackage{tikz}
\usetikzlibrary{shapes,arrows}
\usetikzlibrary{shapes.multipart}
\usepackage{verbatim}

\usepackage{todonotes}

\usepackage{amssymb}
\usepackage[]{amsmath, amsthm, amsfonts,graphicx, amscd,}
\usepackage[all,cmtip]{xy}
\input amssym.def \input amssym
\usepackage{forest}
\usepackage{mdframed}
\usepackage{framed}
\usepackage{csquotes}

\begin{document}

\newtheorem{thm}{Theorem}[section]
\newtheorem{lthm}{Lost Theorem}
\renewcommand{\thelthm}{of Chazy \& Painlev\'e}
\newtheorem*{srp}{The Shelah reflection principle}
\newtheorem{claim}[thm]{Claim}
\newtheorem {fact}[thm]{Fact}
\newtheorem{con}[thm]{Conjecture}

\newtheorem*{thmstar}{Theorem}
\newtheorem*{prob}{Problem}
\newtheorem{prop}[thm]{Proposition}
\newtheorem{cor}[thm]{Corollary}
\newtheorem*{propstar}{Proposition}
\newtheorem {lem}[thm]{Lemma}
\newtheorem*{lemstar}{Lemma}
\newtheorem{conj}[thm]{Conjecture}
\newtheorem{question}[thm]{Question}
\newtheorem*{questar}{Question}
\newtheorem{ques}[thm]{Question}
\newtheorem*{conjstar}{Conjecture}
\newtheorem{fct}[thm]{Fact}
\theoremstyle{remark}
\newtheorem{rem}[thm]{Remark}
\newtheorem{exmp}[thm]{Example}
\newtheorem{cond}[thm]{Condition}
\newtheorem{np*}{Non-Proof}
\newtheorem*{remstar}{Remark}
\theoremstyle{definition}
\newtheorem{defn}[thm]{Definition}
\newtheorem*{defnstar}{Definition}
\newtheorem{exam}[thm]{Example}
\newtheorem*{examstar}{Example}
\newtheorem{assump}[thm]{Assumption}
\newtheorem{Thm}[thm]{Theorem}
\newtheorem{problem}[thm]{Problem}

\theoremstyle{plain}
\newtheorem{thmx}{Theorem}
\renewcommand{\thethmx}{\Alph{thmx}}
\newtheorem{corx}[thmx]{Corollary}

\newtheorem{innerimportantthm}{Theorem}
\newenvironment{importantthm}[1]
  {\renewcommand\theinnerimportantthm{#1}\innerimportantthm}
  {\endinnerimportantthm}
  
\newtheorem{innerimportantcor}{Corollary}
\newenvironment{importantcor}[1]
  {\renewcommand\theinnerimportantcor{#1}\innerimportantcor}
  {\endinnerimportantcor}

\def \ta {\tau_{\mathcal{D}/\Delta}}
\def \D {\Delta}
\def \DD {\mathcal D}

\newcommand{\gen}[1]{\left\langle#1\right\rangle}

\newcommand{\pd}[2]{\frac{\partial #1}{\partial #2}}

\newcommand{\codim}{\text{codim}\,}

\newcommand{\td}{\text{tr.deg.}}
\newcommand{\pp}{\partial }
\newcommand{\pdtwo}[2]{\frac{\partial^2 #1}{\partial #2^2}}
\newcommand{\od}[2]{\frac{d #1}{d #2}}
\def\Ind{\setbox0=\hbox{$x$}\kern\wd0\hbox to 0pt{\hss$\mid$\hss} \lower.9\ht0\hbox to 0pt{\hss$\smile$\hss}\kern\wd0}
\def\Notind{\setbox0=\hbox{$x$}\kern\wd0\hbox to 0pt{\mathchardef \nn=12854\hss$\nn$\kern1.4\wd0\hss}\hbox to 0pt{\hss$\mid$\hss}\lower.9\ht0 \hbox to 0pt{\hss$\smile$\hss}\kern\wd0}
\def\ind{\mathop{\mathpalette\Ind{}}}
\def\nind{\mathop{\mathpalette\Notind{}}}
\numberwithin{equation}{section}

\def\id{\operatorname{id}}
\def\Frac{\operatorname{Frac}}
\def\Const{\operatorname{Const}}
\def\spec{\operatorname{Spec}}
\def\span{\operatorname{span}}
\def\exc{\operatorname{Exc}}
\def\Div{\operatorname{Div}}
\def\cl{\operatorname{cl}}
\def\mer{\operatorname{mer}}
\def\trdeg{\operatorname{trdeg}}
\def\ord{\operatorname{ord}}

\newcommand{\m}{\mathbb }
\newcommand{\mc}{\mathcal }
\newcommand{\mf}{\mathfrak }
\newcommand{\is}{^{p^ {-\infty}}}
\newcommand{\QQ}{\mathbb Q}
\newcommand{\fh}{\mathfrak h}
\newcommand{\CC}{\mathbb C}
\newcommand{\RR}{\mathbb R}
\newcommand{\ZZ}{\mathbb Z}
\newcommand{\tp}{\operatorname{tp}}
\newcommand{\SL}{\operatorname{SL}}
\newcommand{\dcl}{\operatorname{dcl}}
\newcommand{\acl}{\operatorname{acl}}


\title{Not Pfaffian}

\author[J. Freitag]{James Freitag}
\address{James Freitag, University of Illinois Chicago, Department of Mathematics, Statistics,
and Computer Science, 851 S. Morgan Street, Chicago, IL, USA, 60607-7045.}
\email{jfreitag@uic.edu}
\thanks{J. Freitag is partially supported by NSF CAREER award 1945251.}

\date{\today}
\maketitle

\begin{abstract} This short note describes the connection between strong minimality of the differential equation satisfied by an complex analytic function and the real and imaginary parts of the function being Pfaffian. This connection combined with a theorem of Freitag and Scanlon (2017) provides the answer to a question of Binyamini and Novikov (2017). We also answer a question of Bianconi (2016). We give what seem to be the first examples of functions which are definable in o-minimal expansions of the reals and are differentially algebraic, but not Pfaffian. 
\end{abstract}

\section{Why Pfaffian?}
Pfaffian functions were introduced in \cite{hovanskii1980class} where Khovanskii showed that the class has strong finiteness properties. For instance, Khovanskii exhibits an effective bound on the number of zeros in a system of equations involving Pfaffian functions. Later, the strong finiteness properties of this class of functions played an important role in model completeness results for o-minimal expansions of the real field \cite{wilkie1996model,MR1435773}. Any algebraic function is Pfaffian on a suitable domain, but there are two well-known ways to see that a transcendental function is \emph{not Pfaffian}: 
\begin{itemize} 
\item The function is differentially transcendental.\footnote{Sometimes this property is also called being hypertranscendental or transcendentally transcendental \cite{doi:10.1080/00029890.1989.11972282}. The Gamma-function is differentially transcendental by a classical theorem of H\"older \cite{holder1886ueber} and its restriction to $(0,\infty)$ is definable in an o-minimal expansion of the reals \cite{dries2000field}.} 
\item The function violates the strong finiteness properties of the class.\footnote{A popular example is provided by $\sin (x)$ with domain the real line. The violation of the finiteness principles of Khovanskii in this case also yields a violation of o-minimality. In fact, in any example known to the author, examples shown to be violating the finiteness principles of Khovanskii's theorem are strong enough violations to also yield a violation of o-minimality. Because of this, it is not surprising that there is no currently known function which is differentially algebraic, definable in an o-minimal expansion of the real field, and not Pfaffian.}
\end{itemize}

This manuscript introduces a third way of showing a function is not Pfaffian. Our work is based on a very simple idea - Pfaffian functions are built using solutions to certain order one differential equations, while solutions to higher order strongly minimal differential equations can not satisfy order one differential equations. For the notion of $\m C$-Pfaffian, this summary gives nearly a complete indication of the proof that a solution to a higher order strongly minimal equation is not $\m C$-Pfaffian. There are only slightly more complications to describe when considering if the real and imaginary parts of the function are Pfaffian. These complications can already be seen in the work of Macintyre \cite{macintyre2008some}, who shows that meromorphic functions whose derivatives are algebraic have real and imaginary parts which are Pfaffian. Our main result shows that no analog of Macintyre's theorem exists for higher order differential equations.

Throughout, we will use Klein's $j$-function as a motivating example, often stating and explaining our results for the $j$-function before giving the more general result. Besides being a non-misleading example, the $j$-function was the original motivation for this manuscript. Restricted to its standard fundamental domain in the upper half plane, the real and imaginary parts of the $j$-function are definable in the o-minimal expansion of the real field given by the exponential and restricted analytic functions, $\m R_{an, exp}$. As such, the point counting methods of Pila and Wilkie have been applied to certain definable sets associated with the $j$-function. Applying these methods has been a major part of the proof of various open conjectures in diophantine geometry and transcendence theory \cite{PilaAO}. There are several potential sources of ineffectivity in this general approach to the special points problems, but one important source comes from the Pila-Wilkie theorem itself. In general, it is known that the asymptotic bounds of the Pila-Wilkie theorem can not in general be improved (see e.g. page 2 of \cite{pila2006rational} and the discussion in the introduction of \cite{jones2015improving}). In the general o-minimal case, it isn't even entirely clear what sort of effectivity one has in mind\footnote{i.e. On what quantities associated with the definable set are the constants allowed to depend?} without some additional measure of complexity of the sets or functions involved. 

There are, however, improvements to the Pila-Wilkie theorem in important special cases. For instance, \cite[Conjecture 1.11]{pila2006rational} conjectures that for sets definable in $\m R_{exp}$, the bound from the Pila-Wilkie theorem can be improved from $\mc O (h^{\epsilon})$ to $\mc O ( (\log h )^c )$ where the constant $c$ depends on the definable set. Numerous recent works have concentrated, in special cases, on these improved bounds or making the constants of the Pila-Wilkie theorem effective \cite{binyamini2020point, binyamini2017wilkie, binyamini2019complex, butler2012some, jones2012density, jones2020effective, pila2010counting, pila2006note}; each of these works uses the certain powerful finiteness results for Pfaffian functions (sometimes more or less restrictive cases) to obtain effectiveness where the constants depend on the \emph{complexity} of the Pfaffian chain for the functions defining the set. At the same time, attempts have been made to show that effective results from the Pfaffian setting apply to motivational examples coming from number theory \cite{jones2021pfaffian}. 

In the last section of \cite{binyamini2017wilkie}, Binyamini and Novikov isolate two classes of particular interest in diophantine applications, which may be amenable to their approach: elliptic functions and modular functions. The former is known to be amenable to Pfaffian techniques \cite{jones2021pfaffian, macintyre2008some} and is connected to special points conjectures around the Manin-Mumford conjecture, where the effectiveness results have concrete number theoretic consequences. The latter is connected to various number theoretic problems, e.g. the Andr\'e-Oort conjecture. Of it, Binyamini and Novikov write, 
\begin{displayquote}
The modular category currently appears to be more challenging: we have no reason to believe that the j-function is Pfaffian (or definable from Pfaffian functions).
\end{displayquote}

In recent work, Binyamini gives an effective Pila-Wilkie result for \emph{Noetherian functions} \cite{binyamini2019density} in the o-minimal context, a setting which does include the j-function. The results in the Noetherian setting require a relatively compact domain and yield fewer uniformities than the Pfaffian setting. See \cite{jones2020effective} for some discussion of this point. See also \cite{binyamini2020point} for discussion of the difficulties when leaving the Pfaffian setting; this manuscript shows that for the desired applications to certain diophantine problems, the work was truly necessary as theorems from the Pfaffian setting can not apply. 

In the recent work of Armitage \cite{armitage2020pfaffian}, significant effort is required to obtain effective bounds for the zeros of polynomials involving the j-function on its natural (noncompact) domain. As Armitage mentions, it was previously unknown whether the real and imaginary parts of the j-function are Pfaffian. Our results in this manuscript confirm the suspicions of Binyamini and Novikov; the real and imaginary parts of the $j$-function can not be put into a Pfaffian chain. The main theorem of \cite{freitag2017strong} is the essential input to showing this result: 

\begin{thm} \label{strongj} As a definable set in a differentially closed field, the differential equation satisfied by the $j$-function: 
$$
\left(\frac{y''}{y'}\right)' -\frac{1}{2}\left(\frac{y''}{y'}\right)^2 + (y')^2 \cdot \frac{y^2-1968y+2654208}{y^2(y-1728)^2} = 0
 $$
is strongly minimal. 
\end{thm}

\begin{rem}
The equivalent form of the previous theorem that we will use is stated purely in terms of transcendence: a the zero set of a differential equation, $X$, with coefficients in a differential field $K$ is strongly minimal if and only if (1) the equation is irreducible over $K^{alg}$ (as a polynomial in several variables) and (2) given any solution $f$ of $X$ and \emph{any differential field extension\footnote{In proofs, the robustness of this condition is often extremely useful, but we note that there has been significant recent work in showing that this condition follows from the considering only a very particular class of such extensions $F/K$ \cite{freitag2021bounding}.} $F$ of $K$}, $$\text{trdeg}_F \left( F\langle f \rangle \right) =  \text{trdeg}_K \left( K \langle f \rangle \right) \text{ or } 0.$$ 
Here $K \langle f \rangle $ denotes the differential field extension of $K$ generated by $f$. 
\end{rem}

Our main results show that functions which satisfy higher order strongly minimal differential equations can not have real and imaginary parts which algebraic over the elements of any Pfaffian chain. These results are especially pertinent because many of the number theoretic functions to which the Pila-Wilkie theorem has been applied have been shown to be strongly minimal in recent years. For these as well as various other strong minimality results on nonlinear higher order differential equations, see \cite{blazquez2020some, brestovski1989algebraic, casale2020ax, devilbiss2021generic, freitag2017strong, jaoui2019generic, nagloo2017algebraic, poizat1980c}. Besides results for specific equations, as \cite{devilbiss2021generic, jaoui2019generic} indicate, strong minimality is a pervasive condition for nonlinear differential equations of order at least two - it holds \emph{generically} both in the space of constant coefficient equations as well as the space of nonconstant equations. 

The general form of our results is given in Theorem \ref{mainthm}. Following this result we formulate the general problem of determining when a complex analytic function has Pfaffian real and imaginary part using notions from geometric stability theory interpreted in the theory of differentially closed fields.

\section{Pfaffian} 
\begin{defn} Let $f_1, \ldots , f_l $ be real analytic functions on some domain $U \subseteq \m R^n$. We will call $(f_1, \ldots , f_l)$ a \emph{Pfaffian chain} if there are polynomials $p_{ij}(u_1, \ldots , u_n , v_1, \ldots, v_i )$ with coefficients in $\m R$ such that $$\pd{f_i}{x_j}= p_{ij} \left( \bar x, f_1 ( \bar x), \ldots , f_i (\bar x ) \right)$$ 
for $1 \leq i \leq l$ and $1 \leq j \leq n.$ \emph{We call a function Pfaffian} if it can be written as a $\m R$-polynomial with real coefficients in the functions of some Pfaffian chain.
\end{defn} 
We are interested mainly in the connection of the previous notion to certain differential algebraic properties, but the following complex version of the previous notion is more easily connected with our differential algebraic notions: 

\begin{defn} Let $f_1, \ldots , f_l $ be complex analytic functions on some domain $U \subseteq \m C^n$. We will call $(f_1, \ldots , f_l)$ a \emph{$\m C $-Pfaffian chain} if there are polynomials $p_{ij}(u_1, \ldots , u_n , v_1, \ldots, v_i )$ with coefficients in $\m C$ such that $$\pd{f_i}{x_j}= p_{ij} \left( \bar x, f_1 ( \bar x), \ldots , f_i (\bar x ) \right)$$ 
for $1 \leq i \leq l$ and $1 \leq j \leq n.$
We call a function \emph{$\m C$-Pfaffian} if it can be written as a polynomial with $\m C$-polynomial in the functions of some $\m C$-Pfaffian chain.

\end{defn} 
 Our results are insensitive to replacing functions which appear in Pfaffian chains to polynomial (or even algebraic) functions of elements in a Pfaffian chain. 

\section{Not Pfaffian}

\subsection{The $j$-function isn't $\m C$-Pfaffian} We begin by giving a quick and elementary argument for why the the $j$-function can not satisfy a differential equation of the form $j'(z) = f( z, j (z))$ where $f$ is a rational function with coefficients in $\m C.$ Following this, we will show that the $j$-function is not $\m C$-Pfaffian, which generalizes this fact. 

By $SL_2(\m Z)$-invariance, $j(z)=j(z+1).$ It follows also that $j'(z)=j'(z+1)$. Suppose for a moment that there is a rational function with coefficients in $\m C$ so that $j'(z) = f( z, j (z)).$ Now note that $f( z, j (z))=f( z+1, j (z+1)) = f( z+1, j (z)).$ But now since the $j$-function is not algebraic, this equality holds for a generic point $(x,y)$ in affine 2-space, and thus everywhere, so the rational function has the property that $f(x,y)=f(x+1,y)$. This implies that $f$ is only a function of $y$. So, now we have that $$j'(z)=f( j (z) ).$$ But we have that $f(j(z))$ is $SL_2(\m Z)$-invariant, while $j'(z)$ is a quasi-modular form of weight $2.$ That is, $j'(z)$ has the property that if $$\alpha = \begin{pmatrix}
a & b \\
c & d 
\end{pmatrix} \in SL_2 (\m Z)$$
then $j'(\alpha z) = (cz+d)^2 j'(z).$ Thus, it is impossible that $j'(z)=f( j (z) )$. 

Next we argue more generally that the $j$-function can not be $\m C$-Pfaffian using Theorem \ref{strongj}. Let $U \subseteq \m H$ be the standard fundamental domain for the $j$-function: $$U = \left\{ z \in \m H \, : \, |z| \geq 1, -\frac{1}{2}<Re(z) \leq \frac{1}{2}, \text{ and if } Re(z)<0, \text{ then } |z|>1 \right\}$$ 
\begin{thm}
The $j$-function is not algebraic over any $\m C$-Pfaffian chain $(f_1, \ldots , f_l )$ on an open $V \subset U.$
\end{thm}

\begin{proof}
Let $(f_1, \ldots , f_l)$ be a Pfaffian chain on $V \subset V$ of minimal length such that $j(z)$ is algebraic over $\m C ( z, f_1, \ldots , f_l ).$ If $j$ is algebraic over $(f_1, \ldots , f_{l-1}),$ we may shorten the chain contradicting the minimality of $l$. So, we can assume that $j(z)$ is transcendental and interalgebraic with $f_l$ over $\m C ( z, f_1, \ldots , f_{l-1} ).$ By the Pfaffian condition, the field $F=\m C (z, f_1, \ldots , f_{k-1})$ is a $\frac{d}{dz}$-differential field. Again, by the Pfaffian condition, $F(f_l)$ is a $\frac{d}{dz}$-differential field and is a transcendence degree one extension of $F$. Since $f_l$ is interalgebraic with $j$, $\td _F (F \langle j \rangle )=1.$ This is impossible by Theorem \ref{strongj}. 
\end{proof}

One can see that there isn't anything special about the $j$-function in the above argument except that it satisfies a higher order strongly minimal differential equation. So, by the same argument as the previous proof: 

\begin{thm} Let $U \subseteq \m C$ be a connected domain and let $f$ be analytic on $U$. If $f(z)$ is not algebraic and satisfies a strongly minimal differential equation of order two or more, then $f(z)$ is not $\m C$-Pfaffian. 
\end{thm}

\subsection{The real and imaginary parts can't be Pfaffian either} 
\begin{thm}
The real and imaginary parts of the $j$-function are not algebraic over any Pfaffian chain $(f_1, \ldots , f_l )$ on an open $V \subset U.$
\end{thm}
\begin{proof}
Let $z=(x,y)$ denote the real and imaginary coordinates of the complex number $z$, and let $U$ be the domain as above. Suppose there is a Pfaffian chain of shortest length over which both $Re(j(x+iy))$ and $Im(j(x+iy))$ are algebraic. By the same reduction as in the previous subsection, we assume without loss of generality that the final element of the chain $(f_1, \ldots , f_l )$ is transcendental and interalgebraic with one of $Re(j(x+iy))$ or $Im(j(x+iy))$ over the earlier elements. Assume $f_l$ is interalgebraic with $Re(j(x+iy))$ (a similar argument will apply to the case of $Im(j(x+iy))$). To ease notation, we will simply assume that $f_l = Re(j(x+iy))$ and that $Im(j(x+iy))$ appears earlier in the chain. Again, our argument is not at all sensitive to interalgebraicity. 

Now view the $j$-function as a map $j: U \rightarrow \m R^2.$ We have that $Re(j(x+iy))$ is interdefinable (even in the empty language) with $j(z)$ over $\m R (x, y, f_1, \ldots , f_{l-1})$. Note that the reason for this is very simple - we have that $Im(j(x+iy))$ appears in the chain and $j(z) = \left(Re(j(x+iy)), Im(j(x+iy)) \right).$ By the Pfaffian condition, the field $F= \m R (x, y, f_1, \ldots , f_{l-1})$ is a $\frac{d}{dx}$ and a $\frac{d}{dy}$-field. By the Pfaffian condition, with respect to the derivation\footnote{Note that there is nothing special about the $\frac{d}{dx}$ derivation here - one could equally well use $\frac{d}{dy}$.} $\frac{d}{dx}$,
$$F \langle  Re(j(x+iy)) \rangle = F \left(  Re(j(x+iy)) \right)$$ and as previously, we can assume without loss of generality that this extension is not algebraic. That is: $$\td _F (F \langle  Re(j(x+iy)) \rangle )=1.$$ 

By the interdefinability of $Re(j(x+iy))$ and $j(z)$, we have that the transcendence degree of the $\frac{d}{dx}$-field $F \langle j(z) \rangle$ over $F$ is one. By the $\m C$-analyticity of $j(z)$, we have that $j(z)$ satisfies the same differential equation with respect to $\frac{d}{dx}$ as with respect to $\frac{d}{dz}$. Though we are regarding $j(z)$ as a function $\m R^2 \rightarrow \m R^2$, the differential equation is algebraic, and can be expressed by polynomial functions of the real and imaginary parts and their derivatives with respect to $x$ (the derivatives with respect to $y$ are definable from this by the Cauchy-Riemann equations). So, in the $\frac{d}{dx}$-field generated by $Re(j(x+iy))$ or equivalently $j(z)$, we have the functions $\frac{dj}{dz}$ and $\frac{d^2j}{dz^2}$ as functions $\m R^2 \rightarrow \m R^2$. The differential equation, being algebraic, is expressible by polynomial equalities in the real and imaginary parts of these functions.

We now obtain a contradiction to Theorem \ref{strongj}, since the transcendence degree of the differential field generated by any solution to the differential equation satisfied by the $j$-function must be $0$ or $3$. 
\end{proof}

\subsection{The general strongly minimal case} The analysis of the $j$-function in the previous two subsections used only two properties:
\begin{itemize} 
\item The function is complex analytic on some open connected $U \subseteq \m C$. 
\item The function satisfies an algebraic differential equation\footnote{One can actually assume that the function satisfies an equation and finitely many inequations so that the resulting definable set is strongly minimal. For instance, the system $x\cdot x'' - x' =0, \, x' \neq 0$ is one such system, see \cite{MMP} or \cite{poizat1980c} for a proof.} which is strongly minimal and of order $>1.$   
\end{itemize} 
The argument of the previous subsection then gives:

\begin{thm} Let $f(z)$ be an analytic function on some open connected $U \subseteq \m C$. Let $f(z)$ be a non-algebraic solution to an order $h>1$ algebraic differential equation with coefficients in $\m C(z)$ which is strongly minimal. Then one can not build a Pfaffian chain $(f_1, \ldots , f_l)$ with both $Re(f(z))$ and $Im(f(z))$ algebraic over $\m C (z, f_1, \ldots , f_l ).$  
\end{thm}

In the theory of differentially closed fields, the definable closure of $a$ over a differential field $F$ is given by $\dcl (a/F) = F \langle a \rangle$, the differential field generated by $a$ over $F$. The algebraic closure, $\acl (a/F) = F \langle a \rangle ^{alg}$ consists of the algebraic closure (in the sense of fields) of the differential field generated by $a$ over $F.$ The next result generalizes the previous theorem in several ways:

\begin{thm} \label{mainthm}
Let $f(z)$ be a complex analytic function on some open connected subset $U \subset \m C$ and $g(z), h(z) \in \acl (f(z) / \m C (z) )$ such that $g(z)$ satisfies a strongly minimal differential equation over $\m C (z) \langle h(z) \rangle $  with respect to $\frac{d}{dz}$ of order larger than one. Then one can not build a Pfaffian chain with both $Re(f(z))$ and $Im(f(z))$ algebraic over the elements in the chain. 
\end{thm}

\begin{proof}
The previous result is insensitive to replacing $f(z)$ with some function which is algebraic over $\m C (z, f(z))$, because if $f(z)$ is algebraic over $\m C (z, f_1, \ldots , f_l ),$ then so is any element of $\acl (f(z) /\m C(z)).$ Similarly, if $h(z) \in \acl (f(z) / \m C(z) )$, then if $h(z)$ is not Pfaffian, then $f(z)$ can not be Pfaffian. Since $h(z)$ is Pfaffian, we can assume that $h(z)$ appears earlier in the chain. Then the previous arguments apply with $\m C(z) $ replaced by $\m C(z) \langle h(z) \rangle$. 
\end{proof}

\begin{rem}
The previous theorem resolves negatively Question 1 of \cite{Bianconi2016Some}. For instance, the real and imaginary parts of the $j$-function are Noetherian, but on any open domain in $\m C$, these functions can not appear in a Pfaffian chain, by the previous result. 
\end{rem}

\subsection{Minimal analysis} \label{gst}
Even the condition of strong minimality can be slightly weakened - for instance we could instead demand that $g(z)$ is a generic solution of a differential equation $X$ over $\m C (z) \langle h(z) \rangle $ which has Lascar rank one and order larger than one. It is also sufficient to merely demand that $g(z)$ is almost internal to a Lascar rank\footnote{Lascar rank is notion also coming from model theory which has a concrete interpretation in differential fields in terms of transcendence (see e.g. \cite{MMP}). Lascar rank is bounded by Morley rank (in general), but the two can differ \cite{hrushovski1999lascar}.} one type of order greater than one.

Theorem \ref{mainthm} might appear as an extremely special case not likely to be close to contributing to a general characterization of when a function which satisfies an algebraic differential equation has real and imaginary parts which are Pfaffian (similarly $\m C$-Pfaffian). But in this subsection, we will explain why the result is perhaps closer than one might expect.

We begin by giving some definitions from geometric stability theory \cite{GST}, and note that as before we are working in the theory of differentially closed fields of characteristic zero. Throughout the subsection, $p = \tp(a/A)$ will be assumed to be a stationary type. We say that a type is \emph{semiminimal} if it is \emph{almost internal} to a type of Lascar rank $1$. Recall that a stationary type $p = \tp (a/ A)$ is \emph{almost internal} to a minimal type $q$ over $B \supset A$ such that $a \ind _A B$ and there is a sequence $(d_1, \ldots , d_n )$ of $B$-independent realizations of $q$ such that $\acl (Ba) = \acl (B d_1 \ldots d_n )$. The following is a well-known notion from geometric stability theory \cite{moosa2014some}.

\begin{defn}
An \emph{semiminimal analysis of $p = \tp(a/A)$} is a sequence $(a_0, \ldots , a_n)$ such that 
\begin{itemize}
\item $a$ is interdefinable with $a_n$ over $A$, 
\item for each $i$, $a_i \in \dcl (A, a_{i+1} )$,
\item for each $i$, $\tp (a_{i+1} /A a_i)$ is \emph{semiminimal}. 
\end{itemize} 
\end{defn}

Every finite rank type has a semiminimal analysis, and it easily follows from Theorem \ref{mainthm} that if in some analysis of $p=\tp (a/ \m C(z) )$, $(a_0, \ldots, a_n )$, we have a type $tp(a_{i+1} / \m C(z) \langle a_i \rangle )$ which is internal to a strongly minimal type which has order greater than one, then the real and imaginary parts of $a$ are not Pfaffian. Due to the inductive nature of the definition of Pfaffian functions, the following problem can be seen to reduce to the special case of semiminimal types almost internal to types satisfying order one differential equations over differential fields generated by Pfaffian functions. 

\begin{problem} \label{open}
Formulate in differential algebraic terms, necessary and sufficient conditions for the real and imaginary parts of a complex analytic function $f(z)$ to be Pfaffian in terms of the semiminimal analysis of $f(z)$. 
\end{problem}

One can see that the problem must involve in some essential way the differential algebraic equations satisfied by the real and imaginary parts on the given domain, or directly on the domain of complex analytic function. After all, a nontrivial solution to $f'=f$ (a $\m C$-multiple of $e^z$) can not have Pfaffian real and imaginary part on all of $\m C$ (see e.g. \cite{macintyre2008some}). Restricting the domain to those imaginary values in $(-\pi , \pi)$, the complex exponential has real and imaginary parts which are polynomial over the real exponential and restricted trignometric functions, all of which are Pfaffian. The higher order strongly minimal case analyzed above is of a much different nature. For instance, it is robust under domain changes - one can see that the $j$-function can not have its real and imaginary parts in a Pfaffian chain even when restricting to any open subset of the upper half plane.

\begin{defn}
The type $\tp (a /A)$ \emph{admits no proper fibrations} if whenever $c \in \dcl (Aa) \setminus \acl (A)$, we must have $a \in \acl (Ac )$. Minimal types admit no proper fibrations, but there are other (semiminimal) examples, e.g. \cite[example 2.2]{moosa2014some}.
\end{defn}

\begin{prop} \label{moose} \cite[Proposition 2.3]{moosa2014some} Suppose that the stationary type $p = \tp (a/A)$ admits no proper fibrations. Then $p$ is semiminimal and one of two options occurs: 
\begin{enumerate}
    \item $p$ is almost internal to a non locally modular minimal type. 
    \item $a$ is interalgebraic over $A$ with a finite tuple of independent realizations of a locally modular minimal type over $A$. 
\end{enumerate} 
\end{prop}

By refining via fibrations, any finite rank type has a semiminimal analysis $a_0, \ldots , a_n$ in the above sense which is also \emph{reduced} in the sense that for each $i$, $\tp (a_{i+1} /A a_i)$ admits no proper fibrations. Then Problem \ref{open} reduces to the following question with two distinct subcases:

\begin{question}
Let $f(z)$ be a realization of a type $p$ which is $\m C$-analytic on some domain $U \subset \m C.$\footnote{There is always such a realization, by Seidenberg's embedding theorem \cite{seidenberg1958abstract}.} Suppose that $p$ satisfies one of the following:

\begin{enumerate}
    \item Let $p$ be the generic type of some order one differential equation which is internal to the constants.  
    \item Suppose that $p$ does not admit proper fibrations and is interalgebraic with a number of realizations of a locally modular type of order one. The order one type is known to be trivial and $\aleph _0$-categorical \cite{freitag2017finiteness}.  
\end{enumerate}
When is $f(z)$ Pfaffian on $U$? When is $f(z)$ Pfaffian on some open $V \subset U$?
\end{question}

\bibliography{research}{}

\begin{thebibliography}{10}

\bibitem{armitage2020pfaffian}
John Armitage.
\newblock Pfaffian control of some polynomials involving the $ j $--function
  and {W}eierstrass elliptic functions.
\newblock {\em arXiv preprint arXiv:2011.09382}, 2020.

\bibitem{Bianconi2016Some}
Ricardo Bianconi.
\newblock Some model theory of hypergeometric and {P}faffian functions.
\newblock {\em South American Journal of Logic}, 2(2):297--318, 2016.

\bibitem{binyamini2019density}
Gal Binyamini.
\newblock Density of algebraic points on {N}oetherian varieties.
\newblock {\em Geometric and Functional Analysis}, 29(1):72--118, 2019.

\bibitem{binyamini2020point}
Gal Binyamini.
\newblock Point counting for foliations over number fields.
\newblock {\em arXiv preprint arXiv:2009.00892}, 2020.

\bibitem{binyamini2017wilkie}
Gal Binyamini and Dmitry Novikov.
\newblock Wilkie's conjecture for restricted elementary functions.
\newblock {\em Annals of Mathematics}, pages 237--275, 2017.

\bibitem{binyamini2019complex}
Gal Binyamini and Dmitry Novikov.
\newblock Complex cellular structures.
\newblock {\em Annals of Mathematics}, 190(1):145--248, 2019.

\bibitem{blazquez2020some}
David Bl{\'a}zquez-Sanz, Guy Casale, James Freitag, and Joel Nagloo.
\newblock Some functional transcendence results around the {S}chwarzian
  differential equation.
\newblock In {\em Annales de la Facult{\'e} des sciences de Toulouse:
  Math{\'e}matiques}, volume~29, pages 1265--1300, 2020.

\bibitem{brestovski1989algebraic}
Michel Brestovski.
\newblock Algebraic independence of solutions of differential equations of the
  second order.
\newblock {\em Pacific Journal of Mathematics}, 140(1):1--19, 1989.

\bibitem{butler2012some}
Lee~A Butler.
\newblock Some cases of {W}ilkie's conjecture.
\newblock {\em Bulletin of the London Mathematical Society}, 44(4):642--660,
  2012.

\bibitem{casale2020ax}
Guy Casale, James Freitag, and Joel Nagloo.
\newblock Ax-{L}indemann-{W}eierstrass with derivatives and the genus 0
  {F}uchsian groups.
\newblock {\em Annals of Mathematics}, 192(3):721--765, 2020.

\bibitem{devilbiss2021generic}
Matthew DeVilbiss and James Freitag.
\newblock Generic differential equations are strongly minimal.
\newblock {\em arXiv preprint arXiv:2106.02627}, 2021.

\bibitem{dries2000field}
Lou van~ven Dries and Patrick Speissegger.
\newblock The field of reals with multisummable series and the exponential
  function.
\newblock {\em Proceedings of the London Mathematical Society}, 81(3):513--565,
  2000.

\bibitem{freitag2017finiteness}
James Freitag and Rahim Moosa.
\newblock Finiteness theorems on hypersurfaces in partial
  differential-algebraic geometry.
\newblock {\em Advances in Mathematics}, 314:726--755, 2017.

\bibitem{freitag2021bounding}
James Freitag and Rahim Moosa.
\newblock Bounding nonminimality and a conjecture of {B}orovik-{C}herlin.
\newblock {\em arXiv preprint arXiv:2106.02537}, 2021.

\bibitem{freitag2017strong}
James Freitag and Thomas Scanlon.
\newblock Strong minimality and the $ j $-function.
\newblock {\em Journal of the European Mathematical Society}, 20(1):119--136,
  2017.

\bibitem{holder1886ueber}
Otto H{\"o}lder.
\newblock Ueber die eigenschaft der gammafunction keiner algebraischen
  differentialgleichung zu gen{\"u}gen.
\newblock {\em Mathematische Annalen}, 28(1):1--13, 1886.

\bibitem{hovanskii1980class}
AG~Hovanskii.
\newblock On a class of systems of transcendental equations.
\newblock In {\em Soviet Math. Doklady}, volume 255, pages 804--807, 1980.

\bibitem{hrushovski1999lascar}
Ehud Hrushovski and Thomas Scanlon.
\newblock Lascar and {M}orley ranks differ in differentially closed fields.
\newblock {\em The Journal of Symbolic Logic}, 64(03):1280--1284, 1999.

\bibitem{jaoui2019generic}
R{\'e}mi Jaoui.
\newblock Generic planar algebraic vector fields are disintegrated.
\newblock {\em arXiv preprint arXiv:1905.09429}, 2019.

\bibitem{jones2015improving}
Gareth Jones.
\newblock Improving the bound in the {P}ila-{W}ilkie theorem for curves.
\newblock {\em O-Minimality and Diophantine Geometry}, 421:204, 2015.

\bibitem{jones2021pfaffian}
Gareth Jones and Harry Schmidt.
\newblock Pfaffian definitions of {W}eierstrass elliptic functions.
\newblock {\em Mathematische Annalen}, 379(1):825--864, 2021.

\bibitem{jones2012density}
Gareth Jones and Margaret Thomas.
\newblock The density of algebraic points on certain {P}faffian surfaces.
\newblock {\em Quarterly journal of mathematics}, 63(3):637--651, 2012.

\bibitem{jones2020effective}
Gareth Jones and Margaret Thomas.
\newblock Effective {P}ila--{W}ilkie bounds for unrestricted {P}faffian
  surfaces.
\newblock {\em Mathematische Annalen}, pages 1--39, 2020.

\bibitem{macintyre2008some}
Angus Macintyre.
\newblock Some observations about the real and imaginary parts of complex
  {P}faffian functions.
\newblock {\em London Mathematical Society Lecture Note Series},
  1(349):215--224, 2008.

\bibitem{MR1435773}
Angus Macintyre and A.~J. Wilkie.
\newblock On the decidability of the real exponential field.
\newblock In {\em Kreiseliana}, pages 441--467. A K Peters, Wellesley, MA,
  1996.

\bibitem{MMP}
David Marker, Margit Messmer, and Anand Pillay.
\newblock {\em Model theory of fields}.
\newblock A. K. Peters/CRC Press, 2005.

\bibitem{moosa2014some}
Rahim Moosa and Anand Pillay.
\newblock Some model theory of fibrations and algebraic reductions.
\newblock {\em Selecta Mathematica}, 20(4):1067--1082, 2014.

\bibitem{nagloo2017algebraic}
Joel Nagloo and Anand Pillay.
\newblock On algebraic relations between solutions of a generic {P}ainlev{\'e}
  equation.
\newblock {\em Journal f{\"u}r die reine und angewandte Mathematik (Crelles
  Journal)}, 2017(726):1--27, 2017.

\bibitem{pila2006note}
Jonathan Pila.
\newblock Note on the rational points of a pfaff curve.
\newblock {\em Proceedings of the Edinburgh Mathematical Society},
  49(2):391--397, 2006.

\bibitem{pila2010counting}
Jonathan Pila.
\newblock Counting rational points on a certain exponential-algebraic surface.
\newblock In {\em Annales de l'Institut Fourier}, volume~60, pages 489--514,
  2010.

\bibitem{PilaAO}
Jonathan Pila.
\newblock O-minimality and the {A}ndr{\'e}-{O}ort conjecture for {$\mathbf
  C^n$}.
\newblock {\em Ann. of Math.(2)}, 173(3):1779--1840, 2011.

\bibitem{pila2006rational}
Jonathan Pila and Alex Wilkie.
\newblock The rational points of a definable set.
\newblock {\em Duke Mathematical Journal}, 133(3):591--616, 2006.

\bibitem{GST}
Anand Pillay.
\newblock {\em Geometric Stability Theory}.
\newblock Oxford University Press, 1996.

\bibitem{poizat1980c}
Bruno Poizat.
\newblock C'est beau et chaud.
\newblock {\em Groupe d'{\'e}tude de th{\'e}ories stables}, 3:1--11, 1980.

\bibitem{doi:10.1080/00029890.1989.11972282}
Lee~A. Rubel.
\newblock A survey of transcendentally transcendental functions.
\newblock {\em The American Mathematical Monthly}, 96(9):777--788, 1989.

\bibitem{seidenberg1958abstract}
Abraham Seidenberg.
\newblock Abstract differential algebra and the analytic case.
\newblock {\em Proceedings of the American Mathematical Society},
  9(1):159--164, 1958.

\bibitem{wilkie1996model}
Alex~J Wilkie.
\newblock Model completeness results for expansions of the ordered field of
  real numbers by restricted pfaffian functions and the exponential function.
\newblock {\em Journal of the American Mathematical Society}, 9(4):1051--1094,
  1996.

\end{thebibliography}
\bibliographystyle{plain}

\end{document}